\documentclass[12pt]{amsart}
\usepackage{flexisym, amssymb, psfrag,pstricks}
\newtheorem{theo}{Theorem}

\newtheorem{theorem}{Theorem}
\newtheorem{lemma}[theorem]{Lemma}
\newtheorem{prop}[theorem]{Proposition}
\newtheorem{pro}[theo]{Proposition}
\newtheorem{cor}[theorem]{Corollary}

\newtheorem{remark}{Remark}

\renewcommand{\phi}{\varphi}

\newcommand{\R}{\mathbb{R}}

\newcommand{\Z}{\mathbb{Z}}
\newcommand{\N}{\mathbb{N}}

\newcommand{\wZ}{\widetilde{Z}}

\newcommand{\la} {\lambda}
\newcommand{\si}{\sigma}
\newcommand{\al}{\alpha}
\newcommand{\ga}{\gamma}
\newcommand{\om}{\omega}

\newcommand{\eps}{\varepsilon}

\newcommand{\vr}{\varrho}
\newcommand{\vk}{\kappa}
\renewcommand{\epsilon}{\varepsilon}
\renewcommand{\mid}{\,|\,}

\newcommand{\won}{{\boldsymbol 1}}

\newcounter{constante}
\newcommand{\con}[1]{
\immediate\write 1{\noexpand\newlabel{#1}{{\theconstante}{\theconstante}}}
                    c_{\theconstante}
                    \stepcounter{constante}
                   }

\begin{document}

\title[Recurrence and transience of autoregressive processes] {Recurrence and transience of contractive autoregressive processes and related Markov chains}

\author{Martin P.W.\ Zerner} 

\thanks{{\em Address:}
Mathematisches Institut,
Universit\"at T\"ubingen,
Auf der Morgenstelle 10,
72076 T\"ubingen, Germany.
Email: martin.zerner@uni-tuebingen.de
}

\begin{abstract} 
We characterize 
recurrence and transience 
 of nonnegative multivariate autoregressive processes of order one with random contractive coefficient matrix,  of subcritical multitype Galton-Watson branching processes in random environment with immigration, and 
 of the related max-autoregressive processes and general random exchange processes.  Our criterion is given in   terms of the maximal Lyapunov exponent of the coefficient matrix and the cumulative distribution function of the innovation/immigration component.  
\end{abstract}
\maketitle
\section{Introduction}
The classification of irreducible Markov chains as recurrent or transient is one of the fundamental  objectives in the study of Markov chains. 
Scalar {\em nonnegative autoregressive processes} $(X_n)_{n\in\N_0}$ of the form  
\[X_{n}=aX_{n-1}+Y_n,\text{\qquad where $0<a<1$ and $(Y_n)_{n\in\N}$ is i.i.d.,}\] 
and, closely related, {\em subcritical Galton-Watson processes $(Z_n)_{n\in\N_0}$ with immigration $(Y_n)_{n\in\N}$} and average offspring $a\in(0,1)$   are  classical  Markov chains. The study of these processes has a rich history which started more than half a century ago. However, most of the literature on these processes deals only with the positive recurrent case,  i.e.\ the case where there exists a stationary probability distribution.
To the best of our knowledge  there is at present no complete classification in simple terms of recurrence versus transience of these processes although this problem has been investigated for several decades, see \cite{Pak75}, \cite{Pak79}, \cite[Part I]{Kel92}, \cite[p.\ 1196]{GM00}, \cite{ZG04}, \cite{Bau13} and the review below.

In the present article we characterize  recurrence and transience of these processes in terms of $a$ and the cumulative distribution function of $Y_1$. More precisely, we show that either process is 
\begin{equation}
\text{recurrent\qquad iff}\qquad
\sum_{n\ge 0}\prod_{m=0}^nP[Y_1\le ya^{-m}]=\infty
\label{raa}
\end{equation}
 for some $y\in(0,\infty)$,
see Theorem \ref{rtd}. (For the branching process we need to assume  a certain moment condition on the offspring distribution.) 
 Note that the right hand side of (\ref{raa})  can often be easily checked by ratio tests.

 We also extend this result to certain multidimensional  cases  in  random environment by classifying   
nonnegative multivariate autoregressive processes of order one with random contractive coefficient matrix and  subcritical multitype Galton-Watson processes in random environment with immigration. The criterion also applies to  two other related processes, sometimes 
called {\em max-autoregressive process} and {\em general random exchange process}.

We first introduce these four processes and review the existing literature on the subject. 
(The precise definition of recurrence and transience is given in the next section.)

{\bf Autoregressive processes.} Autoregressive  models are among the most widely used stochastic models, see e.g.\ \cite{Kes73}, \cite{BD91}, \cite{MT93},  \cite{BDM16}. We consider nonnegative multidimensional autoregressive processes $X=(X_n)_{n\ge 0}$ of order one (AR(1) processes) with random coefficient matrix, defined as follows. Fix a dimension $d\in\N$. 
Let $Y=(Y_n)_{n\ge 0}$ be a sequence of $[0,\infty)^{d}$-valued random vectors, called {\em innovations}, and let $(A_n)_{n\ge 1}$ be a sequence of $[0,\infty)^{d\times d}$-valued random matrices. Assume that  $(A_n,Y_n)_{n\ge 1}$ is i.i.d.\ and independent of $Y_0$. To avoid cases which are not interesting in the present context we suppose that the support of the law of $Y_1$ is unbounded.
Set
$X_0:=Y_0$ and
\begin{equation}
 \label{rec} X_{n}:=A_{n}X_{n-1}+Y_{n}\quad \text{for $n\ge 1$.}
\end{equation}
Relation (\ref{rec}) is sometimes called a {\em random difference equation} or {\em random affine  recursion}, see also  \cite[p.\ 1]{BDM16}.
Solving this recursion we obtain the explicit expression
\begin{equation}\label{note} 
X_{n}=\sum_{m=0}^nA_{n}A_{n-1}\ldots A_{m+1}Y_{m}\qquad \text{for $n\ge 0$.}
\end{equation}
We only consider the subcritical (contractive) case where the maximal Lyapunov exponent of $A_n$ is strictly less than 0. For convenience we phrase our statements in terms of the negative $\la$ of the Lyapunov exponent, defined as the a.s.\ limit 
\begin{equation}
 \label{lala}\la:=\lim_{n\to\infty}\frac{S_n}n,\quad\text{where}\quad S_n:=-\ln\|A_1\ldots A_n\|.
\end{equation}
(Here  
$S_0:=0$.)
 This limit is known to exist if $E[\log_+\|A_1\|]<\infty$, see \cite[Theorem 2]{FK60}.
For bounds and efficient methods for the computation of $\la$ see e.g.\ \cite{Pol10}.
It has been shown  for the subcritical case ($\la>0$) that under various conditions $X$ is positive recurrent iff
\begin{equation}
  \label{pos} E[\ln_+ \|Y_1\|]<\infty,
 \end{equation}
see e.g.\ \cite[Theorem 1.6. (b)]{Ver79}, \cite[Part III, Theorem (8.5)]{Kel92}, \cite[Corollary 4.1 (b)]{GM00}, \cite[Theorem 2.1.3]{BDM16} for $d=1$ and \cite[Proposition 2]{ZG04} for $d\ge 1$ and constant coefficient matrix $A=A_n$. For a discussion of the multidimensional case with random $A_n$ see 
\cite{Erh14}.
(Note that the equivalence of positive recurrence and the existence of an invariant probability measure, which is well-known for countable state spaces, also holds in this setting, see e.g.\ \cite[Section 6]{Kel06}.)
The case where (\ref{pos}) fails is sometimes referred to as {\em super-heavy tailed}, see \cite{ZG04}.

Among the few works which deal with recurrence versus transience of AR(1) processes are the unpublished preprint \cite[Part I]{Kel92} and \cite{ZG04}. (For  some recent  work with deals with super-heavy tailed innovations see \cite{BI15}.) Both papers treat
one-dimensional AR(1) processes with constant coefficient $A_1=a\in(0,1)$. In \cite[Part I, Theorem (3.1)]{Kel92} Kellerer shows that  $X$ is 
\begin{eqnarray}\label{Ke1}
\text{transient if} &&
\liminf_{t\to\infty}t\cdot P[\ln Y_1>t]>-\ln a\quad\text{ and}\\
\text{recurrent if} &&\limsup_{t\to\infty}t\cdot P[\ln Y_1>t]<-\ln a.\label{Ke2}
 \end{eqnarray}
 In \cite[Theorem 1]{ZG04}  Zeevi and Glynn 
 consider log-Pareto distributed innovations $Y_n$, whose common distribution is given by $P[\ln(1+Y_1)>t]=(1+\beta t)^{-p}$ for some $\beta>0$ and $p>0$.
 For this case they 
 characterize recurrence and transience 
 by showing that $X$ is 
 positive recurrent if $p>1$, null recurrent if $p=1$ and $\beta\ln (1/a)\ge 1$, and transient otherwise.
Note that in this example the distinction between recurrence and transience easily follows from (\ref{raa}) and Raabe's test.

{\bf Branching processes with immigration.}
The classical Galton-Watson model as a basic model for branching populations, see e.g.\ \cite{Har63} and \cite{AN72}, has been extended in various directions, for example by allowing finitely many different types of individuals with different offspring distributions  \cite[Chapter V]{AN72}, by letting the offspring distribution depend on time in a random way \cite[Chapter VI.5]{AN72}, or  by allowing immigration \cite[Chapter VI.7]{AN72}. Following e.g.\ \cite{Key87}, \cite{Roi07}, and \cite{Vat11}, we consider a combination of these three generalizations, namely multitype Galton-Watson branching processes $Z=(Z_n)_{n\ge 0}$ in random environment with immigration.

We postpone the precise definition of $Z$ to the next section and first give an informal description of the model. Let $d\in\N$ and  $(Y_n)_{n\ge 0}$ be as above and let us assume for the moment that all $Y_n$ are $\N_0^d$-valued.
There are $d$ different types of individuals, enumerated by $1,\ldots,d$. The $i$-th component of the $\N_0^d$-valued random variable $Z_n$ is the number of individuals of type $i$ present in generation $n$. 
Given $Z_{n-1}$, the $n$-th generation $Z_n$ is obtained as follows. Each member of generation $n-1$ gets independently of the other members of that generation a random number of children of the $d$ different types. The distribution of the number of children of a certain type may depend on the type of the parent. It may also depend in an i.i.d.\ way, called the random environment, on the number $n-1$ of the generation.  
The $n$-th generation consists of the children of the individuals of the previous generation and additional immigrants of type $1,\ldots,d$, whose numbers are given by $Y_n$.

This process $Z$ is closely related to AR(1) processes in the following way. Define the $(i,j)$-th entry of the matrix $A_n$ as 
the conditional expectation of the number of children of type $i$ in generation $n$ of a parent of type $j$ given the  random environment.
Then the conditional expected value of $Z$ given the random environment and given the numbers of immigrants satisfies the recursion (\ref{rec}) and is therefore an AR(1) process, see (\ref{hust2}) below.

 The process $Z$ is called subcritical iff $\la>0$, where $\la$ is defined as in (\ref{lala}).
 As for AR(1) processes, positive recurrence of a subcritical $Z$ is related to the validity of  (\ref{pos}), see e.g.\  \cite{Qui70},
 \cite[Corollary 2]{FW71}, \cite[Theorem A]{Pak79},  \cite[Theorem 3.3]{Key87},  \cite{Roi07}, and  \cite{Vat11}.

 In the one-dimensional case  the results in the literature concerning  the distinction between recurrence and transience of $Z$ are more complete than the corresponding results  for autoregressive processes. Pakes considers  in \cite{Pak75} and \cite{Pak79}  subcritical single-type  processes  with immigration in an environment which is constant in time. He gives several sufficient conditions  for recurrence   or transience in terms of generating functions and provides  several examples. 

For subcritical single-type branching processes in random environment Bauernschubert derives in \cite[Theorems 2.2, 2.3]{Bau13} conditions similar to (\ref{Ke1}) and (\ref{Ke2}). She shows that  under suitable assumptions
  $Z$ is
\begin{eqnarray}\label{BB1}
\text{transient if} &&
\liminf_{t\to\infty}t\cdot P[\ln Y_1>t]>-E[\ln A_1]\quad\text{ and}\\
\text{recurrent if} &&\limsup_{t\to\infty}t\cdot P[\ln Y_1>t]<-E[\ln A_1].\label{BB2}
 \end{eqnarray}

For a different but similar model in continuous time, Li, Chen, and Pakes \cite[Theorem 3.3 (ii)]{LCP12} give a necessary and sufficient criterion for recurrence and transience, also in  terms of generating functions.  Unfortunately, ``it is not easily applicable in specific cases" \cite[p.\ 136]{LCP12}. This raises the question whether a modification of our criterion (\ref{raa})  also holds for that model.

{\bf Max-autoregressive processes.} By replacing the sum in (\ref{rec}) with the maximum we obtain the process 
$M=(M_n)_{n\ge 0}$ defined by $M_0:=Y_0$ and 
\begin{equation}\label{den}
M_{n}:=\max\{A_{n}M_{n-1},Y_{n}\},\quad n\ge 1.
\end{equation}
Here the maximum is taken for each coordinate of $\R^d$  separately. 
Such processes have been studied e.g.\ in  \cite{Gol91} and \cite{RS95} and  are sometimes called {\em max-autoregressive}.
They appear naturally in our proof. 
If $d=1$ then similarly to (\ref{note}),
$M_{n}=\max_{m=0}^nA_n\ldots A_{m+1}Y_{m}$  for all $n\ge 0$.
For general dimension $d\ge 1$ we have
\begin{equation}
 \label{note3} X_n\ge M_n\ge N_n:=\max_{m=0}^nA_n\ldots A_{m+1}Y_{m},
\end{equation}
where the inequalities hold for each of the $d$ components. 
 We are not aware of any results in the literature on the classification of recurrence versus transience of max-autoregressive processes.

{\bf General random exchange processes.} 
These are one-dimensional processes $R=(R_n)_{n\ge 0}$  which have been studied e.g.\ in \cite{HN76} and are defined as follows.
Let $(W_n)_{n\ge 0}$ be a sequence of nonnegative random variables with unbounded support and let $(T_n)_{n\ge 1}$ be a sequence of real-valued random variables such that $(T_n,W_n)_{n\ge 1}$ is i.i.d.\ and independent of $W_0$. Set $R_0:=W_0$ and 
\begin{equation}R_{n}:=\max\{R_{n-1}-T_{n},W_{n}\},\quad n\ge 1.
 \label{gep}
\end{equation}
The starting point of our investigation was the following classification of recurrence and transience of $R$ in the special case where $T_n$ is a positive constant (random exchange process).
To the best of our knowledge this observation  was first made by Kesten  in the appendix to \cite{Lam70}, where it was phrased in terms of long range percolation. Later it was stated in a more general form in terms of Markov chains by Kellerer \cite[pp.\ 268,269]{Kel06}.
\begin{pro}\label{pro}{\bf (Random exchange process; Kesten, Kellerer)} 
Let $W_n,n\ge 1,$ be i.i.d.\ $\N_0$-valued random variables satisfying $P[W_1=0]>0$. 
Then the state 0 is recurrent for the Markov chain $R$ satisfying $R_{n}:=\max\{R_{n-1}-1,W_{n}\}$ iff  
\begin{equation}\label{kk}
 \sum_{n\ge 0}\prod_{m=0}^nP[W_1\le m]=\infty.
\end{equation}
\end{pro}
\begin{proof} It is well-known, see e.g.\ \cite[Theorem 6.4.2]{Dur10}, that $0$ is recurrent iff 
 $\sum_{n\ge 1}P[R_n=0]=\infty$. Solving the recursion  with initial state 0 yields
 $R_n=\max_{m=1}^n(W_m-n+m)$ for all $n\ge 1$. Since $(W_n)_{n\ge 1}$ is i.i.d.\ we have for all $n\ge 1$,
 \[P[R_n=0]=\prod_{m=1}^nP[W_m\le n-m]=\prod_{m=0}^{n-1}P[W_1\le m].\]
The claim follows.
\end{proof}

The significance of Proposition \ref{pro}  in the present context is that on a heuristic level one can easily deduce from it in several steps the recurrence/transience criterion for our processes $X,Z$ and $M$ introduced above. However, our actual proof will not follow these steps. 

{\em Step 1.}  If $R$ satisfies only the more general recursion $R_{n}=\max\{R_{n-1}-c,W_{n}\}$ 
for some  constant $c\in\N$ then  the event $\{W_1\le m\}$ in (\ref{kk}) has to be replaced by  $\{W_1\le mc\}$.
  
 {\em Step 2.} If we do not require the minimum $y$ of the support of $W_1$ to be 0 
 then  the event $\{W_1\le mc\}$ in Step 1 has to be replaced by  $\{W_1\le y+mc\}$.
The same holds for any $y$ satisfying $P[W_1\le y]>0$.

{\em Step 3.}  It is easy to guess but much harder to prove that for the general random exchange process satisfying (\ref{gep}) and $E[T_1]>0$ the event 
$\{W_1\le y+mc\}$ from Step 2 has to be replaced by $\{W_1\le y+mE[T_1]\}$, see Corollary \ref{grep} below.
 
 {\em Step 4.} The process $e^R$ is a one-dimensional max-autoregressive process $M$ which satisfies the recursion $M_{n}=\max\{A_{n}M_{n-1},Y_{n}\}$ with $A_n=e^{-T_n}$ and $Y_n=e^{W_n}$. It follows from Step 3 that if $y$ is such that $P[Y_1\le y]>0$ then $M$ should be recurrent iff $\sum_{n\ge 0}\prod_{m=0}^nP[Y_1\le ye^{-mE[\ln A_1]}]$ is infinite.

 {\em Step 5.} By the strong law of large numbers, $\la=-E[\ln A_1]$ if $d=1$. Thus for multi-dimensional max-autoregressive processes one should replace $-E[\ln A_1]$ in Step 4 by $\la$ and get that for all $y$ satisfying  $P[\|Y_1\|\le y]>0$, $M$ is recurrent iff 
 \begin{equation}\sum_{n\ge 0}\prod_{m=0}^nP[\|Y_1\|\le ye^{m\la}]=\infty.
  \tag{RR}
 \end{equation}

 {\em Step 6.} It is a well-known phenomenon (max-sum-equivalence) that the sum of heavy tailed random variables tends to be comparable to the largest summand. Thus one might expect that the recurrence criterion for $M$ derived in Step 5 also applies to $X$ and, due to the relation between $X$ and $Z$ described above,  to $Z$ as well.
 
 That the conclusion of this heuristics is indeed true under suitable conditions is the content of our main results, Theorem \ref{rtd} and Theorem \ref{rtr}.

\begin{remark}\label{rm}{\rm 
  Since positive recurrence implies recurrence, (\ref{pos}) should imply (RR). This implication can easily be derived directly as follows.
Note that (\ref{pos}) is equivalent to $\sum_{n\ge 0}P\left[\|Y_1\|> ye^{m\la}\right]<\infty$. This in turn is equivalent to $\prod_{m\ge 0} P\left[\|Y_1\|\le ye^{m\la}\right]>0$, which implies (RR).
 }
\end{remark}

Let us now describe how the present article is organized. In the next section we provide additional definitions
and collect some elementary statements. 
In Section \ref{ce}  we first treat the special case of constant,  deterministic environments because in this case our proof is shorter and requires weaker assumptions than in the genuinely random case. 
We also provide an application to so-called frog  processes with geometric lifetimes.
 The general case of random environments is dealt with in Section \ref{re}, where we also give an application to random walks in random environments perturbed by cookies.
In the appendix we collect some general bounds which we need in Section \ref{re} but were not able to find in the literature.

\section{Preliminaries}

\subsection{Notation}
The $\ell_p$-vector norms ($1\le p\le\infty$) on $\R^d$ and their associated matrix norms are denoted by $\|\cdot\|_{p}$. We abbreviate $\|\cdot\|_\infty$ by $\|\cdot\|$. Recall that for a matrix $A$, $\|A\|$ is the maximum row sum and  $\|A\|_1$ is the maximum column sum.
The $i$-th coordinate of a vector $x$ is denoted by $[x]_i$ and 
the $(i,j)$-th entry of a matrix $A$ by $[A]_{i,j}$.
For $x,y\in[0,\infty)^d$ we write $x\le y$ (or $y\ge x$) iff $[x]_i\le [y]_i$ for all $i=1,\ldots,d$.
By  $c_1,c_2,\ldots$ we mean  suitable strictly positive and finite constants which may depend on other constants.

\subsection{Branching processes}
While branching processes are most often defined and studied in terms of generating functions we prefer to use a different, but equivalent definition which allows us to couple the branching process in a natural way to the AR(1) process introduced above, see (\ref{hust}) and (\ref{hust2}) below.

Let $d\ge 1$ and let $\Phi$ be the set of all measurable functions $\psi:[0,1]\to\N_0^d$. An {\em environment} for a multitype Galton-Watson branching process is a sequence $(\psi_n)_{n\ge 1}=((\psi_n^j)_{j=1,\ldots,d})_{n\ge 1}\in (\Phi^d)^{\N}$. Here $\psi_n$ determines the reproduction behavior of the individuals in the $(n-1)$-st generation, namely, if $U$ is distributed uniformly on $[0,1]$ then $P[\psi_n^j(U)=(x_1,\ldots,x_d)]$ is interpreted as the probability that an individual of type $j$ in the $(n-1)$-st generation gets $x_i$ children of type $i$, $i=1,\ldots,d$.

Let $\Psi=(\Psi_n)_{n\ge 1}=((\Psi_n^j)_{j=1,\ldots,d})_{n\ge 1}$ be an i.i.d.\ sequence of  $\Phi^d$-valued random variables, called the {\em random environment} for the branching process, and let $Y=(Y_n)_{n\ge 0}$ be a sequence of $[0,\infty)^{d}$-valued random vectors  such that $(\Psi_n,Y_n)_{n\ge 1}$ is i.i.d.\ and independent of $Y_0$. The vector $\lfloor Y_n\rfloor$ of integer parts of the components of $Y_n$ gives the numbers of {\em immigrants} of the $d$ possible types who join the population at time $n$.
Moreover, let $(U_{m,n,k}^j)_{0\le m< n; 1\le k; 1\le j\le d}$ be an i.i.d.\ family of random variables which are distributed uniformly on $[0,1]$. Assume that this family is independent of $\Psi$ and $Y$.
Set 
\begin{equation}\label{aij}
\xi_{m,n,k}^{i,j}:=[\Psi_n^j(U_{m,n,k}^j)]_i.
\end{equation}
We interpret $\xi_{m,n,k}^{i,j}$ as the (random) number of children of type $i$ of the $k$-th individual of type $j$ in generation $n-1$ whose ancestors immigrated at time $m$, provided that there are at least  $k$ individuals of this kind.
Define for all $m\ge 0$ the process $(B_{m,n})_{n\ge m}$ by setting $B_{m,m}:=\lfloor Y_m\rfloor$  and 
\[ 
B_{m,n}:=\left(\sum_{j=1}^d\sum_{k=1}^{[B_{m,n-1}]_j}\xi_{m,n,k}^{i,j}\right)_{i=1,\ldots,d},\ n\ge m+1.
\]
 Here $[B_{m,n}]_j$ stands for the number of individuals of type $j$ at time $n$ which descended from  the individuals who immigrated at time $m$.
Then the process $(Z_n)_{n\ge 0}$
defined by $Z_0:=\lfloor Y_0\rfloor$ and 
 \begin{equation}\label{ZZ2}
 Z_{n}=\sum_{m=1}^nB_{m,n},\quad n\ge 1,
\end{equation}
(or any other process with the same distribution) is called a {\em branching process with immigration $\lfloor Y\rfloor$ in the random environment $\Psi$.}
Here $[Z_n]_j$ is the number of individuals of type $j$ present at time $n$. 
The random matrix 
 \[A_n:=\left(E\left[\xi_{0,n,1}^{i,j}\mid \Psi\right]\right)_{i,j=1,\ldots,d}\]
contains at position $(i,j)$ the expected value, given the environment,  of the number of children of type $i$ of a  member of type $j$ of the $(n-1)$-st generation. 
As above in the definition of the AR(1) process $X$, the sequence  $(A_n,Y_n)_{n\ge 1}$ is i.i.d.\ and independent of $Y_0$.
It is well-known that for all $0\le m\le n$,
 \[E[B_{m,n}\mid\Psi,Y]=A_n\ldots A_{m+1}\lfloor Y_m\rfloor,\]
see for example 
\cite[Chapter II, (4.1)]{Har63}.
It follows from (\ref{note}), (\ref{ZZ2}),  and $\lfloor Y\rfloor\le Y$  that 
for all $n\ge 0$ a.s.
\begin{eqnarray}E[Z_n\mid \Psi, Y]&\le& X_n
\qquad\text{and} \label{hust}
\\
 E[Z_n\mid \Psi, Y]&=& X_n\qquad\text{if $Y_m\in\N_0^d$ a.s.\ for all $m\ge 0$.}\label{hust2}
\end{eqnarray}

\subsection{Recurrence and transience}\label{rat}
We use the notion of recurrence and transience of
$[0,\infty)^d$-valued Markov chains which was 
introduced by Kellerer in
\cite{Kel06} in a more general setting.
 Let $\mathcal H$ be the set of  functions from $[0,\infty)^d$ to $[0,\infty)^d$ which are monotone with respect to the partial order $\le$.
 Then a  $[0,\infty)^d$-valued Markov chain $(V_n)_{n\ge 0}$ is called {\em order-preserving} if it
fulfills a recursion of the form
 $V_{n}=H_n(V_{n-1})$ for an i.i.d.\ sequence $(H_n)_{n\ge 1}$ of $\mathcal H$-valued random variables which is independent of the initial value $V_0$. 
 Observe that all four processes $X,Z,M$, and $R$ defined above are order-preserving.

  Let $\pi$ be the transition kernel of such an order-preserving Markov chain.
  Then $\pi$ (and any Markov chain with transition kernel $\pi$) is 
called {\em irreducible} for the state space $[0,\infty)^d$ iff  for any $x\in [0,\infty)^d$ there is some $n\ge 0$ such that
$P[V_n\ge x]>0$, where $V=(V_n)_{n\ge 0}$ is a Markov chain with kernel $\pi$ starting at 0, see \cite[Definition 1.1]{Kel06}.
\begin{prop}\label{irr}
Let $K\in\N$ be such that $P[A_1\ldots A_K\in(0,\infty)^{d\times d}]>0$. Then the processes $X,Z,$ and $M$ are irreducible for the state space $[0,\infty)^d$.
\end{prop} 
\begin{proof} Without loss of generality we assume that a.s.\ $Y_1\in\N_0^d$.
 Let $\mu$ be the minimum of the  entries of the matrix $A_{K+1} A_K\ldots A_2$ and choose $\eps>0$ such that $P[\mu\ge\eps]>0$. Then we have for all $x\in[0,\infty)^d$ due  (\ref{hust2})  and (\ref{note3}) that
 \begin{eqnarray*}P[E[Z_{K+1}\mid \Psi,Y]\ge x]&=&P[X_{K+1}\ge x]\ge P[M_{K+1}\ge x]\ \ge\ P[N_{K+1}\ge x]\\
 &\ge& P[A_{K+1} A_K\ldots A_2Y_1\ge x]
 \ \ge\ P\left[\mu\|Y_1\|\ge\|x\|,\mu\ge\eps\right]\\
 &\ge& P[\|Y_1\|\ge \|x\|/\eps] P[\mu\ge \eps]
 \end{eqnarray*}
 by independence. Therefore, $P[V_{K+1}\ge x]>0$ for all $V\in\{Z,X,M\}$. 
\end{proof}
If $\pi$ is irreducible then $\pi$  (and any Markov chain with transition kernel $\pi$) is called  
{\em recurrent} iff there exists $b\in(0,\infty)$ such that 
\begin{equation}\label{forall}\sum_{n\ge 0}P[\|V_n\|\le b]=\infty,
\end{equation}
 where $V$ is a Markov chain with kernel $\pi$ starting at 0. In fact, the initial state is not important here. Condition (\ref{forall}) holds either for all Markov chains with transition kernel $\pi$ or for none, see \cite[Definition 2.5]{Kel06}. 
 A Markov chain $V$ is recurrent iff there is a finite $b$ such that a.s.\ $\|V_n\|\le b$ infinitely often.
If $\pi$ is not recurrent then it is called {\em transient}. 
Transience is equivalent to the almost sure  divergence of the Markov chain in all components to $\infty$, see \cite[Section 2]{Kel06}.
(For the definition and characterization of positive recurrence in this context see \cite[Section 6]{Kel06}.)

In order to deduce from the recurrence of one process the recurrence of another process  we will need to infer from the divergence of a series of the form
$\sum_{n\ge 0}a_n$ the divergence of another series $\sum_{n\ge 0}b_n$. Sometimes we will do this by showing either $\sup_na_n/b_n<\infty$ or $\inf_nb_n/a_n>0$. Sometimes we shall use the following lemma instead.
\begin{lemma}\label{eth}
For $n\ge 0$ let $U_n$ and $V_n$ be $\R^d$-valued random variables.   Assume that there are $b,c>0$ such that $\sum_nP[\|U_n\|\le b]=\infty$ and  $E[\|V_n\|; \|U_n\|\le b]\le cP[\|U_n\|\le b]$ for all $n\ge 0$. Then $\sum_nP[\|V_n\|\le b]=\infty$.
\end{lemma}
\begin{proof}
By Markov's inequality,
\begin{eqnarray*}
P[\|V_n\|\le 2c]&\ge& P[\|V_n\|\le 2c, \|U_n\|\le b]\\
&=&P[\|U_n\|\le b]-P[\|V_n\|> 2c, \|U_n\|\le b]\\
&\ge& P[\|U_n\|\le b]-\frac{E[\|V_n\|; \|U_n\|\le b]}{2c}\ge \frac{P[\|U_n\|\le b]}2,
\end{eqnarray*}
which is not summable in $n$ by assumption. 
\end{proof}

\section{Constant environment}\label{ce}
Recall that a matrix $A\in[0,\infty)^{d\times d}$ is called {\em primitive} iff there is a $K\in\N$ such that $A^K\in(0,\infty)^{d\times d}$. 
\begin{theorem} {\em \bf (Subcritical case, constant environment)} \label{rtd} 
Assume that there is a  primitive matrix $A$ with spectral radius $\vr<1$ such that a.s.\ $A_n=A$ for all $n\ge 1$. 
Let $y\in(0,\infty)$ be such that $P[\|Y_1\|\le y]>0$. Then the following three assertions are equivalent.
\begin{align}
  \tag{XR} &\text{The autoregressive processes $X$ is recurrent.}\\
  \tag{MR} &\text{The max-autoregressive process $M$ is recurrent.}\\
 \tag{RC} &\sum_{n\ge 0}\prod_{m=0}^{n} P\left[\|Y_1\|\le y\vr^{-m}\right]=\infty.
 \end{align}
 If we assume in addition that 
there is a $\psi\in\Phi^d$ such that a.s.\ $\Psi_n=\psi$ for all $n\ge 1$ 
and that $E[\xi_{0,1,1}^{i,j}\ln\xi_{0,1,1}^{i,j}]<\infty$ for all $i,j\in\{1,\ldots,d\}$
then {\rm (XR), (MR)}, and {\rm (RC)} are equivalent to the following assertion.
 \begin{align}
  \tag{ZR} \text{The branching process with immigration $Z$ is recurrent.}
 \end{align}
\end{theorem}
\begin{proof}
By Proposition \ref{irr}, $X,Z$, and $M$ are irreducible since $A$ is primitive. When checking (XR), (MR), and (ZR) we assume without loss of generality that $Y_0$ has the same distribution as $Y_n,\ n\ge 1$, since the initial state does not matter, see Section \ref{rat}.
Recall from Perron-Frobenius theory, see e.g.\   \cite[Appendix, Theorem 2.3]{KT75}, that there is a matrix $H\in(0,\infty)^{d\times d}$ such that 
\begin{equation}\lim_{n\to\infty}\vr^{-n}A^n=H.
 \label{limi}
\end{equation}
We consider the following auxiliary conditions. (Recall (\ref{note3}) for the definition of $N_n$.)
\begin{itemize}
 \item[(NR)] There exists $b\in(0,\infty)$ such that 
$\sum_{n\ge 0}P[\|N_n\|\le b]=\infty$.
\item[(RC')] There exists $b\in(0,\infty)$ such that 
 ${\sum_{n\ge 0}\prod_{m=0}^{n} P\left[\|Y_1\|\le b\vr^{-m}\right]=\infty.}$
\end{itemize}
We shall prove the equivalence of the conditions (XR), (MR), (ZR), (RC), (NR), and (RC') as indicated in the following diagram.
\[\begin{pspicture}(0,0.7)(4,3.3)
\rput(0,2){(XR)}
\rput(1,2){$\Longleftarrow$}
\rput(2,2){(NR)}
\rput(2,1){(ZR)}
\rput{30}(1,2.5){$\Longrightarrow$}
\rput{330}(1,1.5){$\Longrightarrow$}
\rput(3,2){$\Longleftrightarrow$}
\rput(3,1){$\Longrightarrow$}
\rput(4,2){(RC')}
\rput(4,1.5){$\Updownarrow$}
\rput(4,1){(RC)}
\rput(2,3){(MR)}
\rput(2,2.5){$\Downarrow$}
\end{pspicture}
\]
\underline{(XR)$\Rightarrow$(MR)$\Rightarrow$(NR):} This  implication follows from 
(\ref{note3}). \\
\underline{(NR)$\Rightarrow$(XR):}
Due to (\ref{limi}),
\begin{equation}\label{su}
 \sum_{n\ge 0}\|A^n\|<\infty.
 \end{equation}
 Therefore and by (NR), one can choose $b\in(0,\infty)$ large enough such that 
\begin{equation}\label{sun}
P[\|A^mY_1\|\le b]\ge 1/2\quad\text{for all $m\ge 0$} 
\quad\mbox{and}\quad
\sum_{n\ge 0}P\left[\|N_n\|\le b\right]=\infty.
\end{equation} 
Again by (\ref{su}),  $T:=\inf\{m\ge 0: \|A^mY_1\|\le b\}$ is a.s.\ finite. Moreover,  
\begin{equation}\|N_n\|=\max_{m=0}^n\|A^{n-m}Y_{m}\|\quad\text{ for all $n\ge 0.$}
\label{nn}
\end{equation}
 Therefore, by (\ref{note}) and since $Y$ is i.i.d.\ we have for all $n\ge 0$ that 
\begin{eqnarray*}
\lefteqn{E\left[\|X_n\|\mid \|N_n\|\le b\right]\le 
\sum_{m=0}^{n}E\left[\|A^{n-m}Y_{m}\|\ \Bigg|\ \bigcap_{i=0}^n\{\|A^{n-i}Y_{i}\|\le b\}\right]}\\
&=&\sum_{m=0}^{n}E\left[\|A^{n-m}Y_{m}\|\mid \|A^{n-m}Y_{m}\|\le b\right]\ =\ 
\sum_{m= 0}^{n}\frac{E[\|A^mY_1\|; \|A^mY_1\| \le b]}{P[\|A^mY_1\|\le b]}\\
&\overset{(\ref{sun})}\le& 2E\left[\sum_{m\ge T} \|A^mY_1\|\right]\ =\ 2E\left[\sum_{m\ge 0} \|A^{m+T}Y_1\|\right]
\ \le\ 2E\left[\left(\sum_{m\ge 0}\|A^m\|\right)\|A^T Y_1\|\right]\\
&\le& 2b\sum_{m\ge 0}\|A^m\|\overset{(\ref{su})}<\infty.
\end{eqnarray*}
Applying Lemma \ref{eth} to $(U,V)=(N,X)$ we obtain the claim (XR).

\underline{(NR)$\Leftrightarrow$(RC'):}  Since $Y$ is i.i.d.\ and (\ref{nn}) holds, (NR) is satisfied iff there is a $b\in(0,\infty)$ such that
 $\prod_{m=0}^{n} P\left[\|A^mY_1\|\le b\right]$ is not summable.
Due to (\ref{limi}) there exist $k,L\in\N$ such that $k^{-1}\vr^m\le [A^m]_{i,j}\le k\vr^m$ for all $m\ge L$ and all $i,j=1,\ldots,d$.
Therefore, $k^{-1}\vr^m\|Y_1\|\le \|A^mY_1\|\le dk\vr^m\|Y_1\|$
for all $m\ge L$.
This implies the claim.

\underline{(RC')$\Leftrightarrow$(RC):} The implication $\Leftarrow$ is trivial. For the reverse implication  let $b\in(0,\infty)$  be according to (RC') and $k\in\N_0$ be such that $y\vr^{-k}\ge b$. Then
\begin{eqnarray*}\lefteqn{
 \sum_{n\ge 0}\prod_{m=0}^{n} P\left[\|Y_1\|\le y\vr^{-m}\right]\ \ge\ 
 \sum_{n\ge k}\prod_{m=0}^{n} P\left[\|Y_1\|\le y\vr^{-m}\right]}\\
 &=&\left(\prod_{m=0}^{k-1} P\left[\|Y_1\|\le y\vr^{-m}\right]\right)\sum_{n\ge 0}\prod_{m=0}^{n} P\left[\|Y_1\|\le \left(y\vr^{-k}\right)\vr^{-m}\right]=\infty.
\end{eqnarray*}

\underline{(XR)$\Rightarrow$(ZR):}
Let $b\in(0,\infty)$ be such that $\sum_nP[\|X_n\|\le b]=\infty$. 
Note  that for all $x\in[0,\infty)^d$,
$\|x\|\le \|x\|_1=x_1+\ldots+x_n\le d\|x\|.$
Therefore,
\begin{eqnarray*}
E[\|Z_n\|; \|X_n\|\le b]&\le&E\left[\|Z_n\|_1; \|X_n\|\le b\right]=\|E[E[Z_n; \|X_n\|\le b\mid \Psi, Y]]\|_1\\
&=&\|E[E[Z_n\mid \Psi, Y]; \|X_n\|\le b]\|_1\ \overset{(\ref{hust})}\le\  \|E[X_n; \|X_n\|\le b]\|_1\\
&\le& E[d\|X_n\|; \|X_n\|\le b]
\ \le\ bdP[\|X_n\|\le b].
\end{eqnarray*}
Lemma \ref{eth} applied to $(U,V)=(X,Z)$ implies (ZR).

\underline{(ZR)$\Rightarrow$(RC):} 
Denote by $q_{j,\ell}$ the probability that a given individual of type $j$ does not have any descendants $\ell$ generations later, i.e.\ with a slight abuse of notation, $q_{j,\ell}:=P[B_{0,\ell}=0\mid \lfloor Y_1\rfloor =e_j]$, where $e_j\in\Z^d$ is the $j$-th standard unit vector. 
Due to the moment assumption on $\xi_{0,1,1}^{i,j}$ and  \cite[Theorems 2 (3.6) and 4]{JS67}, $\vr^{-k}(1-q_{j,k})$ tends as $k\to\infty$ for all $j=1,\ldots,d$, to a strictly positive and finite limit. In particular, there are $c,\ell\in\N$ such that 
\begin{equation}
 \label{da}
  c(1-q_k)\ge \vr^k\ \text{ for  all $k\ge \ell$, where }\ q_k:=\max_{j=1}^dq_{j,k}.
\end{equation}
Set $\wZ_n:=\sum_{m=0}^{n-\ell}B_{m,n}$ for $n\ge \ell$. By (ZR) and subadditivity there is some $z\in\N_0^d$ such that $\sum_{n}P[Z_n=z]=\infty$. 
Then on the one hand by the Markov property and independence,
\begin{equation}
\sum_{n\ge \ell}P[\wZ_{n}=0]\ge \sum_{n\ge \ell}P[Z_{n-\ell}=z]\prod_{j=1}^dq_{j,\ell}^{[z]_j}=\infty\label{gsf}
\end{equation}
by our choice of $z$ and $\ell$.
On the other hand, since the processes $(B_{m,m+n})_{n\ge 0}$, $m\ge 0$, are i.i.d., we have for all $n\ge \ell$,
\begin{eqnarray*}\nonumber\lefteqn{
 P[\wZ_n=0]= P\left[\bigcap_{m=0}^{n-\ell}\{B_{m,n}=0\}\right]
 \ =\ \prod_{k=\ell}^{n}E[P[B_{0,k}=0|Y_1]]=\prod_{k=\ell}^{n}E\left[\prod_{j=1}^dq_{j,k}^{\lfloor [Y_1]_j\rfloor}\right]}\\
 & \le& \prod_{k=\ell}^{n}E\left[q_{k}^{\|\lfloor Y_1\rfloor\|}\right]
 \ = \prod_{k=\ell}^{n}\int_0^1P\left[q_{k}^{\|\lfloor Y_1\rfloor\|}\ge t\right]\ dt \overset{(\ref{da})}\le 
 \prod_{k=\ell}^{n}\int_0^1 G(c(-\ln t)\vr^{-k})\ dt,
\end{eqnarray*}
 where $G$ is the cumulative distribution function of $\|\lfloor Y_1\rfloor\|$.
Note that $G(y)>0$ and set $\bar G:=1-G$.
 Then by the above for all $n\ge \ell$,
\begin{eqnarray}\nonumber
\frac{P[\wZ_n=0]}{\prod_{m=\ell}^nG(y\vr^{-m})}
&\le& \prod_{m=\ell}^{n}\int_0^1\frac{G\left(c(-\ln t)\vr^{-m}\right)}{G(y\vr^{-m})}\ dt\\ \nonumber
&=&\prod_{m=\ell}^{n}\left(1+\int_0^1\frac{\bar G(y\vr^{-m})-\bar G\left(c(-\ln t)\vr^{-m}\right)}{G(y\vr^{-m})}\ dt\right)\\
&\le&\nonumber
\exp\left(G(y)^{-1}\sum_{m=\ell}^{n}\int_0^1\left(\bar G(y\vr^{-m})-\bar G\left(c(-\ln t)\vr^{-m}\right)\right)_+\ dt\right)\\
&=&\exp\left(G(y)^{-1}\int_0^{\exp\left(-y/c\right)}\sum_{m=\ell}^n\bar G(y\vr^{-m})-\bar G\left(c(-\ln t)\vr^{-m}\right)\ dt\right). \label{col}
\end{eqnarray}
We set $f(t):=\left(\ln (c/y)+\ln(-\ln t)\right)/(-\ln\vr)$ and use the telescopic form of the sum in (\ref{col}) for estimating this sum for all $t\in (0,e^{-y/c})$ from above by
\begin{eqnarray}&&\sum_{m=\ell}^{n\vee(\ell+\lceil f(t)\rceil)}\bar G(y\vr^{-m})-\bar G\left(c(-\ln t)\vr^{-m}\right)\nonumber\\
&\le& \lceil f(t)\rceil+1+\label{umb}
\sum_{m=\ell}^{(n-\lceil f(t)\rceil)\vee\ell}\bar G(y\vr^{-(m+\lceil f(t)\rceil )})-\bar G\left(c(-\ln t)\vr^{-m}\right)\le f(t)+2
\end{eqnarray}
since all the differences in (\ref{umb}) are nonpositive. Since $\int_0^{\exp\left(-y/c\right)}f(t)\ dt<\infty$, the right-hand side of (\ref{col}) is  bounded from above uniformly in $n$.
Therefore, (\ref{gsf}) implies that 
 $\prod_{m=\ell}^nG(y\vr^{-m})$ is not summable.  Since 
 $G(x)\le P[\|Y_1\|\le 2x]$
for all $x\ge 1$,
(RC) follows.
\end{proof}
\subsection{An application to mortal frog processes} For a survey on frog processes we refer to \cite{Pop03}. 
The following application is related to \cite[Theorem 4.3]{Pop03}.
Let $(Y_n)_{n\ge 0}$ be an i.i.d.\ sequence of $\N_0$-valued random variables. Put on each $n\ge 0$ a number $Y_n$ of sleeping frogs. Fix $p,r\in (0,1)$.
Wake up the frogs at 0 (if there are any). Once woken up, every frog performs a nearest-neighbor random walk, jumping independently of everything else with probability $r$ to the right and with probability $1-r$ to the left, until it dies after an independent number of steps which is geometrically distributed with parameter $1-p$ and may be 0.
Whenever a frog visits a site with sleeping frogs those frogs are woken up as well and start their own independent lives.
\begin{theorem}\label{frog}
Let $y\in(0,\infty)$ be such that $P[Y_0\le y]>0$.
Then the following statements are equivalent.
\begin{align}\label{f1}
 &\text{Almost surely only finitely many different frogs visit 0.}\\ \label{f2}
 &\text{Almost surely only finitely many frogs are woken up.}\\ \label{f3}
 & \sum_{n\ge 0}\prod_{m=0}^nP\left[Y_0\le y\vr^{-m}\right]=\infty,\quad\text{where}\quad
 \vr:=\frac{1-\sqrt{1-4p^2r(1-r)}}{2p(1-r)}.
  \end{align}
\end{theorem}
\begin{proof}
Let $a_\pm\in(0,1)$ be the probability that a frog which starts at $0$ 
ever hits $\pm1$ before it dies.

\underline{(\ref{f2})$\Rightarrow$(\ref{f1}):} This implication is obvious.

\underline{(\ref{f2})$\Leftrightarrow$(\ref{f3}):} 
By conditioning on the first step we see that $a_+$ satisfies $a_+=pr+p(1-r)a_+^2$ and get $a_+=\vr$.
Assign to each frog an a.s.\ finite trajectory which the frog will follow once it has been woken up.
For any $0\le m\le n$ let $B_{m,n}$ be the number of frogs originally sleeping at $m$ whose trajectories reach the site $n$. Then for all $m\ge 0$, $B_{m,m}=Y_m$ and $(B_{m,m+k})_{k\ge 0}$ is a Galton-Watson branching processes with offspring distribution Bernoulli($a_+$). 
Moreover, the processes $(B_{m,m+k})_{k\ge 0}, m\ge 0,$ are independent. 
Hence, if we denote by $Z_n$, $n\ge 0$, the total number of frogs originating in $\{0,1,\ldots,n\}$ whose trajectories visit $n$ then $(Z_n)_{n\ge 0}$ is a subcritical branching process with immigration. By Theorem \ref{rtd}, $(Z_n)_{n\ge 0}$ is recurrent iff (\ref{f3}) holds. 
On the other hand, $(Z_n)_{n\ge 0}$ is recurrent iff there is a.s.\ an $n\ge 1$ such that $Z_n-Y_n=0$, i.e.\ iff there is a site $n$ which is never visited, which is equivalent to  (\ref{f2}).

\underline{$(\neg(\ref{f2})\wedge \neg(\ref{f3}))\Rightarrow\neg$(\ref{f1}):} Since the frogs jump between nearest neighbors, $\neg(\ref{f2})$ implies that with positive probability all frogs are woken up. Moreover, as shown in Remark \ref{rm}, $\neg(\ref{f3})$ implies $E[\ln_+ Y_0]=\infty$ and hence a.s.\ 
$\sum_{n\ge 0}Y_na_-^n=\infty$, see e.g.\ \cite[Theorem 5.4.1]{Luk75}. Since $a_-^n$ is the probability that a frog starting at $n$ ever reaches $0$, $\neg(\ref{f1})$ follows from the Borel Cantelli lemma.
\end{proof}
\section{Random environment}\label{re}
For the case of genuinely random environments we need the following boundedness assumptions.
 Denote for  $j=1,\ldots,d$ by $\mathcal V^j$ the covariance matrix of the vector $(\xi_{0,1,1}^{i,j})_{i=1,\ldots,d}$ given $\Psi$.
\begin{enumerate}
\item[(BD1)] There exists $\ga_1\in\N$ such that a.s.\ $\|A_1\|\le\ga_1$.
\item[(BD2)] There exists $\ga_2\in\N$ such that a.s.\ $\|\mathcal V^1\|,\ldots,\|\mathcal V^d\|\le\ga_2.$
\end{enumerate}
We shall also use the following joint primitivity assumption, which is stronger than the one in Proposition \ref{irr}.
\begin{itemize}
 \item[(PR)] 
There exist  $\kappa>0$ and $K\in\N$ such that a.s.\ $A_1\ldots A_K\in[\kappa,\infty)^{d\times d}.$
\end{itemize}
We need the following rather mild regularity condition on the distribution of $Y_1$.
\begin{equation}
 \tag{REG} \begin{array}{l}\text{There exists $\beta\in(2/3,1)$ such that ${\displaystyle \lim_{x\to\infty}x^{\beta}  P[\|Y_1\|>e^x]=0}$}\\
\text{or ${\displaystyle \liminf_{x\to\infty}x P[\|Y_1\|>e^x]> \la}$.}
\end{array}
\end{equation}
Note that (REG) holds if $P[\|Y_1\|>e^x]$ varies regularly as $x\to\infty$.
\begin{theorem} {\em \bf (Subcritical case, random environment)} \label{rtr}
Assume {\em (BD1), (PR)}, {\em (REG)} and $\la>0$. Let $y\in(0,\infty)$ be such that $P[\|Y_1\|\le y]>0$. Then the following three statements are equivalent.
\begin{align}
  \tag{XR} &\text{The autoregressive processes $X$ is recurrent.}\\
  \tag{MR} &\text{The max-autoregressive process $M$ is recurrent.}\\
  \tag{RR}&\sum_{n\ge 0}\prod_{m=0}^nP[\|Y_1\|\le ye^{m\la}]=\infty.
 \end{align}
If we assume in addition {\em (BD2)} then {\em (XR), (MR),} and {\em (RR)} are equivalent to the following statement.
\begin{align}
  \tag{ZR} \text{The branching process with immigration $Z$ is recurrent.}
 \end{align}
\end{theorem}
For the proof of Theorem \ref{rtr} we denote the cumulative distribution function of $\ln\|Y_1\|$ by $F$ and its tail by $\bar F:=1-F$.
\begin{lemma}\label{quo}
Assume {\em (REG)}, let $\beta\in(2/3,1)$ be accordingly and $\al\in(0,\beta)$.
 Suppose that for all $\eps>0$ there exists $b_\eps\in(0,\infty)$ such that 
$\sum_{n\ge 0}\prod_{i=0}^nF(b_\eps+(\la +\eps)i)=\infty.$
  Then   
$\lim_{x\to\infty}x^{\beta}\bar F(x)=0$
and therefore
$E[(\ln_+\|Y_1\|)^{\al}]<\infty.$ 
\end{lemma}
\begin{proof}
  Raabe's test implies that for all $\mu>1$ and $\eps>0$, $F(b+(\la+\eps)i)\ge 1-\mu/i$ for infinitely many $i$. Therefore $\liminf_x x\bar F(x)\le \la+\eps$. Letting $\eps\searrow 0$ yields $\liminf_x x\bar F(x)\le \la$.
The statement now follows from (REG). 
\end{proof}
\begin{proof}[Proof of Theorem \ref{rtr}]
By Proposition \ref{irr} and (PR), $X,Z$, and $M$ are irreducible. As in the proof of Theorem \ref{rtd} we  assume that $Y_0$ has the same distribution as $Y_n,\ n\ge 1$.

Denote by $\mathcal A\subseteq [0,\infty)^{d\times d}$ the support of $A_1$.
Due to (BD1) and (PR) the assumptions of Lemma \ref{dd} from the appendix are satisfied.

Recall (\ref{note3}) and the auxiliary condition (NR) from the proof of Theorem \ref{rtd}.
Fix $\beta\in(2/3,1)$  according to (REG) and $\al\in(1-\beta/2,\beta)$ and consider the following three additional auxiliary statements. 
\begin{itemize}
 \item[(RR')] There exists $b\in(0,\infty)$ such that $\sum_{n\ge 0}\prod_{m=0}^{n} P\left[\|Y_1\|\le be^{m\la}\right]=\infty.$
 
 \item[(R$+$)] $E[(\ln_+\|Y_1\|)^{\al}]<\infty$ and there exists $b\in(0,\infty)$ such that \\
  $\sum_{n\ge 0}\prod_{i=0}^nF(b+\la(i+i^\al))=\infty.$
\item[(R$-$)] There exists $b\in(0,\infty)$ such that\ $\sum_{n\ge 0}\prod_{i=0}^nF(b+\la(i-i^\al))=\infty.$
\end{itemize}
We shall prove the equivalence of the conditions (XR), (MR), (ZR), (NR), (RR), (RR'), (R$+$), and (R$-$) as indicated in the following diagram.
\[\begin{pspicture}(0,0.7)(6,3.3)
\rput(0,2){(XR)}
\rput(1,2){$\Longrightarrow$}
\rput(2,2){(ZR)}
\rput(2,1){(R$-$)}
\rput{330}(1,1.5){$\Longleftarrow$}
\rput{30}(3,1.5){$\Longleftarrow$}
\rput(3,2){$\Longrightarrow$}
\rput(3,1){$\Longrightarrow$}
\rput(4,2){(R$+$)}
\rput(4,1.5){$\Uparrow$}
\rput(4,1){(RR')}
\rput(5,1){$\Longleftrightarrow$}
\rput(6,1){(RR)}
\rput(1,3){(MR)}
\rput(3,3){(NR)}
\rput{45}(.6,2.5){$\Rightarrow$}
\rput{315}(3.4,2.5){$\Rightarrow$}
\rput(2,3){$\Longrightarrow$}
\end{pspicture}
\]
The proofs of (XR)$\Rightarrow$(MR)$\Rightarrow$(NR), of (XR)$\Rightarrow$(ZR) and of (RR)$\Leftrightarrow$(RR') are the same as for the corresponding statements of Theorem \ref{rtd}.

\underline{(NR)$\Rightarrow$(R$+$):}
Set
$N_n':=\max_{i=0}^nA_1\ldots A_{i}Y_{i+1}$ for all $n\ge 0$ and note that $N'_n$ has the same distribution as $N_n$ since 
 $(A_n,\ldots,A_1,Y_n,\ldots,Y_0)$ has the same distribution as $(A_1,\ldots,A_n,Y_1,\ldots,Y_{n+1})$. Therefore, by (NR) there is a $c\in(0,\infty)$ such that
\begin{eqnarray*}
\infty&=&\sum_{n\ge 0} P\left[ \|N_n'\|\le c\right]\
=\ \sum_{n\ge 0} E\left[P\left[\forall i=0,\ldots,n: \|A_1\ldots A_{i}Y_{i+1}\|\le c\mid (A_k)_{k\ge 1}\right]\right].
\end{eqnarray*}
Since $(A_k,Y_k)_{k\ge 1}$ is independent, $(Y_k)_{k\ge 1}$ is independent given $(A_k)_{k\ge 1}$. Therefore,
\begin{eqnarray*}
\infty&=&\sum_{n\ge 0} E\left[\prod_{i=0}^n P\left[\|A_1\ldots A_{i}Y_{i+1}\|\le c\ \bigg|\ (A_k)_{k\ge 1}\right]\right]
\\
&\overset{(\ref{wiede})}\le &\sum_{n\ge 0} E\left[\prod_{i=0}^n P\left[\|A_1\ldots A_{i}\| \|Y_{i+1}\|\le \con{qwe}\ \bigg|\ (A_k)_{k\ge 1}\right]\right]\
 =\ R(\con{rtz}),
\end{eqnarray*}
where $R(t):=\sum_{n\ge 0}E\left[\prod_{i=0}^nF\left(t+S_i\right)\right].$
For any  function $g:\N_0\to[0,\infty)$ with $g(0)=0$ let 
$T_g:=\inf\{n\ge 0 \mid \forall i> n: S_i\le \la i+g(i)\}$. Then for all $i\ge 1$,
\begin{equation}
P[T_g=i]\le
P\left[S_i\ge \la i+g(i)\right]
\le\ \con{coo}\exp\left(-\con{co}g(i)^2/i\right)\label{deep}
\end{equation}
due to Lemma \ref{mcd}.
Moreover, 
for all $a\in(0,\infty)$ with $F(a)>0$,
\begin{eqnarray}\nonumber
q(a,g)&:=&\sup_{n\ge 0}E\left[\prod_{i=0}^n\frac{F(a+S_i)}{F(a+\la i + g(i))}\right]\le
E\left[\prod_{i\ge 0}\left(\frac{F(a+S_i)}{F(a+\la i+ g(i))}\vee 1\right)\right]\\
&=& E\left[\prod_{i= 1}^{T_g}\left(\frac{F(a+S_i)}{F(a+\la i + g(i))}\vee 1\right)\right]
\ \le\   \label{st}
E\left[\prod_{i=1}^{T_g}\frac{1}{F(a+\la i)}\right].
\end{eqnarray}
 For $\eps>0$ and $i\in\N_0$ let $g_\eps(i):=\eps i$. Then (\ref{st}) yields $q(a,g_\eps)\le E\left[F(a)^{-T_{g_\eps}}\right]$.
Consequently, since $T_{g_\eps}$ 
has some finite exponential moment due to (\ref{deep}), there is for all $\eps>0$ some $b_\eps\ge c_{\ref{rtz}}$ such that $q(b_\eps, g_\eps)<\infty$. 
Since  $R(b_\eps)\ge R(c_{\ref{rtz}})=\infty$, the assumptions of 
Lemma \ref{quo} are satisfied. This lemma yields the first statement in (R$+$) and the existence of $a\ge c_{\ref{rtz}}$ such that 
\begin{equation}\label{lal}
  \bar F( a+\la i)\le (\la i)^{-\beta}\wedge 1/2\quad \mbox{ for all $i\ge 1$.}
\end{equation}
Consider now $g(i):=\la i^\al$.
We obtain from (\ref{st}) that
\begin{eqnarray}\nonumber
q(a,g)&\le & E\left[\prod_{i=1}^{T_g}\left(1+\frac{\bar F(a+\la i)}{F(a+\la i)}\right)\right]\ \le\ E\left[\exp\left(\sum_{i=1}^{T_g}\frac{\bar F( a+\la i)}{F( a+\la i)}\right)\right]\\ \label{kuh}
 &\overset{(\ref{lal})}\le&  E\left[\exp\left(2\sum_{i=1}^{T_g}\bar F( a+\la i)\right)\right]\ \overset{(\ref{lal})}\le\   
E\left[\exp\left(2\la^{-\beta}\sum_{i= 1}^{T_g} i^{-\beta}\right)\right]\\ \label{cow}
&\le &
E\left[\exp\left(2\la^{-\beta}{T_g}^{1-\beta}\right)\right]
\ \overset{(\ref{deep})}\le \  1+c_{\ref{coo}}\sum_{i\ge 1}\exp\left(2\la^{-\beta}i^{1-\beta}-c_{\ref{co}}i^{2\al-1}\right),
\end{eqnarray}
which is finite, since $1-\beta<2\al-1$. Since $R(a)\ge R(c_{\ref{rtz}})=\infty$ this implies (R$+$). 

\underline{(ZR)$\Rightarrow$(R$+$):} 
Set $P_\Psi[\cdot]:=P[\, \cdot\mid\Psi]$.
Denote by $q_{\Psi,j,m,n}$ the probability that in the environment $\Psi$ a given individual of type $j$ who immigrated at time $m$ does not have any descendants at time $n$, i.e.\ with a slight abuse of notation, $q_{\Psi,j,m,n}=P_\Psi[B_{m,n}=0\mid \lfloor Y_m\rfloor =e_j]$.
Proposition \ref{ex} and (BD2) yield that for all $j=1,\ldots,d,$ a.s.\
\begin{equation}\label{quak}
\left(1-d\|A_n\ldots A_{m+1}\|\right)_+\le
q_{\Psi,j,m,n}\le 1-\con{dee}\frac{\|A_n\ldots A_{m+1}\|}{\sum_{k=m+1}^n\|A_n\ldots A_k\|}=:q_{\Psi,m,n}.
\end{equation}
Since $\la>0$, there is an integer  $\ell$ large enough such that $P[d\|A_\ell\ldots A_1\|<1]>0$.
Set $\wZ_n:=\sum_{m=0}^{n-\ell}B_{m,n}$ for $n\ge \ell$.
Since  there is a $b\in(0,\infty)$ such that $\sum_nP\|Z_n\|\le b]=\infty$, there is a $z\in\N_0^d$ such that $\sum_nP[Z_n=z]=\infty$.
Hence
\begin{equation}\label{cez}
P[\wZ_n=0]\ge P[\wZ_n=0, Z_{n-\ell}=z]
=P[Z_{n-\ell}=z] E\left[\prod_{j=1}^dq_{\Psi,j,0,\ell}^{[z]_j}\right]
\end{equation}
since $\Psi$ is i.i.d.. Due to the lower bound in (\ref{quak}) and our choice of $\ell$ the expected value in (\ref{cez}) is strictly positive.
This and our choice of $z$ imply 
\begin{equation}\label{neu}
\sum_{n\ge \ell}P[\wZ_n=0]=\infty.
\end{equation}
On the other hand, for all $n\ge \ell$,
\begin{equation}
 P[\wZ_n=0]=E\left[P_\Psi\left[\bigcap_{m=0}^{n-\ell}\{B_{m,n}=0\}\right]\right]
 =E\left[\prod_{m=0}^{n-\ell}P_\Psi[B_{m,n}=0]\right].\label{bvb}
\end{equation}
Denote the cumulative distribution function of $\ln\|\lfloor Y_1\rfloor\|$ by $L$. 
By independence,
 \begin{eqnarray}\nonumber
  P_\Psi[B_{m,n}=0]&=&E_\Psi\left[P[B_{m,n}=0\mid\Psi,Y]\right]=E_\Psi\left[\prod_{j=1}^dp_{\Psi,j,m,n}^{\lfloor [Y_m]_j\rfloor}\right]\overset{(\ref{quak})}\le 
  E_\Psi\left[q_{\Psi,m,n}^{\|\lfloor Y_m\rfloor\|}\right]\\
  &=&\int_0^1P_\Psi \left[q_{\Psi,m,n}^{\|\lfloor Y_m\rfloor\|}\ge t\right]\ dt=\int_0^1 
  L\left(\ln\left(\frac{\ln t}{\ln q_{\Psi,m,n}}\right)\right)\ dt.\label{hsv}
 \end{eqnarray}
 Now let  $g:\N_0\to[0,\infty)$ be  such that $g(0)=0$.
 Then by (\ref{bvb}) and (\ref{hsv})  for all $n\ge \ell$  and all $a\in(0,\infty)$ such that $L(a)\ge 1/2$,
\begin{equation}\label{uni}
\frac{P[\wZ_n=0]}{\prod_{m=\ell}^nL(a+\la m+g(m))}
= E\left[\prod_{i=0}^{n-\ell}\int_0^1\frac{L\left(\ln\left(\frac{\ln t}{\ln q_{\Psi,i,n}}\right)\right)}{L(a+\la(n-i)+g(n-i))}\ dt\right].
\end{equation}
 Since $(A_1,\ldots,A_n)$ has the same  distribution as $(A_n,\ldots,A_1)$,
 $(q_{\Psi,0,n},\ldots,q_{\Psi,n-1,n})$ has the same  distribution as $(q'_{\Psi,0,n},\ldots,q'_{\Psi,n-1,n})$,
 where 
\begin{equation}\label{lon}q'_{\Psi,i,n}:=1-c_{\ref{dee}}\frac{\|A_1\ldots A_{n-i}\|}{\sum_{k=i+1}^n\|A_1\ldots A_{n+1-k}\|}
\le 1-c_{\ref{dee}}\frac{\|A_1\ldots A_{n-i}\|}{\si} 
\end{equation}
and $\si:=\sum_{k\ge 1}\|A_1\ldots A_{k}\|$.  Let $r_{\Psi,i}:=
\exp\left(-c_{\ref{dee}}\|A_1\ldots A_{i}\|/\si\right)$, $f_{\Psi,i}(t):=\ln\left(\frac{\ln t}{\ln r_{\Psi,i}}\right)$, and $t_{\Psi,i}:= r_{\Psi,i}^{\exp(a+\la i + g(i))}$.
By (\ref{lon}), $q'_{\Psi,i,n} \le r_{\Psi,n-i}.$
Therefore, the right hand side of (\ref{uni}) can be estimated from above by
\begin{eqnarray}\nonumber\lefteqn{
E\left[\prod_{i=0}^{n-\ell}\int_0^1\frac{L\left(f_{\Psi,i}(t)\right)}{L(a+\la i+g(i))}\ dt\right]}\\ \nonumber
&=&E\left[\prod_{i=0}^{n-\ell}\left(1+\int_0^1\frac{\bar L(a+\la i+g(i))-\bar L\left(f_{\Psi,i}(t)\right)}{L(a+\la i+g(i))}\ dt\right)\right]\\
&\le& \nonumber
E\left[\exp\left(2\sum_{i\ge 0}\int_0^1\left(\bar L(a+\la i+g(i))-\bar L\left(f_{\Psi,i}(t)\right)\right)_+\ dt\right)\right]\\
&=&\nonumber
  E\left[\exp\left(2\sum_{i\ge 0}\int_0^{t_{\Psi,i}}\left(\bar L(a+\la i+g(i))-\bar L\left(f_{\Psi,i}(t)\right)\right)_+\ dt\right)\right]\\
&\le& \label{rain}
 E\left[\exp\left(2\sum_{i\ge 0}t_{\Psi,i}\bar L(a+\la i)\right)\right].
 \end{eqnarray}
Let 
$ T_g:=\inf\{m\ge 0\mid \forall i>m:  S_i\le \la i+g(i)/2-\ln\si\}.
$
Then $t_{\Psi,i}\le \exp\left(-c_{\ref{dee}}e^{a+g(i)/2}\right)$ for all $i> T_g$.
Hence the quantity in (\ref{rain})  is less than or equal to
\begin{equation}
  \label{dog}
 \exp\left(2+2\sum_{i\ge 0}\exp\left(-c_{\ref{dee}}e^{a+g(i)/2}\right)\right)E\left[\exp\left(2\sum_{i= 1}^{T_g}\bar L(a+\la i)\right)\right].
\end{equation}
In summary, it follows from (\ref{neu}) and (\ref{uni})--(\ref{dog}) that if the expression in (\ref{dog}) is finite then $\sum_{n\ge 0}\prod_{i=0}^n L(a+\la i+g(i))=\infty$.
We shall use this fact for two functions $g$ of the form 
$g(i)=\zeta i^\eta$ with  $\zeta>0$ and $0<\eta\le 1$. 
For any such $g$ we have due to Lemma \ref{W} for all $i\ge 0$,
\begin{equation}
P[T_g=i]\le
\con{et}\exp\left(-\con{et2}i^{2\eta-1}\right).
\label{deep2}
\end{equation}
First, fix $\eps>0$ and let $g(i):=\eps i$ for $i\in\N_0$. In this case the term in (\ref{dog}) can be bounded from above by 
 $\con{cat}E\left[\exp\left(2\bar L(a)T_{g}\right)\right]$, which is finite for large enough $a$  since $T_{g}$ 
has some finite exponential moment due to (\ref{deep2}).
Therefore, the assumptions of Lemma \ref{quo} are satisfied with $\lfloor Y_1\rfloor$ instead of $Y_1$.
Consequently,  $a$ can be chosen large enough such that (\ref{lal}) holds with $\bar L$ instead of $\bar F$.
Consider now $g(i):=\la i^\al$. Then the expression in (\ref{dog}) is finite
due to the same computation as in (\ref{kuh}) and (\ref{cow}), where we use (\ref{deep2}) instead of (\ref{deep}).
This proves (R$+$) with $L$ instead of $F$.  Since 
 $L(x)\le F(x+1)$
for all $x\ge 0$, (R$+$) follows. 

\underline{(R$+$)$\Rightarrow$(R$-$):}
Define $g_\pm(t):=b+\la(t\pm t^{\al})$ for $t\in[1,\infty)$.
Note that both functions $g_+$ and $g_-$ are strictly increasing.
It suffices to show that the following quantities are finite.
\begin{eqnarray*}
\sup_{n\ge 1}\prod_{i=1}^n\frac{F(g_+(i))}{F(g_-(i))}=\prod_{i\ge 1}\left(1+\frac{F(g_+(i))-F(g_-(i))}{F(g_-(i))}\right)\le 
\prod_{i\ge 1}\left(1+\frac{F(g_+(i))-F(g_-(i))}{F(b)}\right).
\end{eqnarray*}
Therefore, it is enough to show that
$F(g_+(i))-F(g_-(i))$ is summable in $i$. 
Let $m\in\N$ be large enough such that
$t\ge 2t^{\al}$ for all $t\ge m$ and set $\eta:=\ln\| Y_1\|$. Then 
\begin{eqnarray}\nonumber
\sum_{i\ge m}F(g_+(i))-F(g_-(i))&=&\sum_{i\ge m}E[\won_{(g_-(i), g_+(i)]}(\eta)]\ =\ E\left[\sum_{i\ge m}\won_{[g_+^{-1}(\eta),g_-^{-1}(\eta))}(i)\right]\\ \label{bel}
&\le & E[g_-^{-1}(\eta)-g_+^{-1}(\eta); g_-^{-1}(\eta)\ge m]+1.
\end{eqnarray}
However, on the event $\{g_-^{-1}(\eta)\ge  m\}$, by definition of $g_\pm$ and our choice of $m$,
\begin{eqnarray*}
 g_-^{-1}(\eta)-g_+^{-1}(\eta)&=&
\left(g_-^{-1}(\eta)\right)^{\al}+\left(g_+^{-1}(\eta)\right)^{\al}\le 2\left(g_-^{-1}(\eta)\right)^{\al}\\
&\le& 2\left(2g_-^{-1}(\eta)-2(g_-^{-1}(\eta))^{\al}\right)^{\al}=2\left(\frac{2(\eta-b)}\la\right)^{\al}.
\end{eqnarray*}
Therefore, the expression in (\ref{bel}) is finite due to $E[\eta^\al]<\infty$.

\underline{(R$-$)$\Rightarrow$(RR'):} This implication follows from monotonicity.

\underline{(RR')$\Rightarrow$(R$+$):} The first part of (R$+$) follows from monotonicity and Lemma \ref{quo}. The second part follows from monotonicity.

\underline{(R$-$)$\Rightarrow$(XR):}
Let $b$ be according to (R$-$).
Due to (\ref{lala}) and $\la>0$, we have $(S_i-\la i^\al/2)\to\infty$ a.s.\ as $i\to\infty$.
Therefore, there is $b'\ge b$ such that 
 $P[D]>0,$ where
 $D:=\{F(b'+S_i-\la i^\al/2)>1/2\text{ for all $i\ge 0$}\}.$
For $i\ge 0$ set $\mu_i:=\exp\left(b'-\la i^\al/2\right)$. Then $\mu:=\sum_{i\ge 0}\mu_i<\infty$.
Recall (\ref{note}) and set for all $n\ge 0$,
$X_n':=\sum_{i=0}^nA_1\ldots A_{i}Y_{i+1}.$
Then  for each $n$,
 $X_n$ has the same distribution as $X'_n$.
Therefore, it suffices  to show that
$\sum_{n\ge 0}P[\|X_n'\|<\mu]=\infty$. Hence we estimate
\begin{eqnarray*}\lefteqn{
P[\|X_n'\|<\mu] \ge P\left[\sum_{i=0}^n \|A_1\ldots A_{i}\|\ \|Y_{i+1}\|<\mu\right]}\\
&\ge&
E\left[P\left[\bigcap_{i=0}^n \{\ln\|Y_{i+1}\|<\ln \mu_i+S_i\}\Bigg| (A_k)_{k\ge 1}\right]\right]\
=\ E\left[\prod_{i=0}^nF(\ln m_i+S_i)\right].
\end{eqnarray*}
Set $T:=\inf\{n\ge 0\mid \forall i> n: S_i\ge \la (i-i^\al/2)\}.$ Then
\begin{eqnarray*}\lefteqn{
\inf_{n\ge 0}E\left[\prod_{i=0}^n\frac{F(\ln m_i+S_i)}{F(b'+\la(i-i^\al))}\right]
\ge
E\left[\prod_{i\ge 1}\left(\frac{F(\ln m_i+S_i)}{F(b'+\la(i-i^\al))}\wedge 1\right); D
\right]}\\
&=&
E\left[\prod_{i=1}^{T}\left(\frac{F(b'+S_i-\la i^\al/2)}{F(b'+\la(i-i^\al))}\wedge 1\right);D
\right]\ge E[2^{-T};D]>0
\end{eqnarray*}
since  $P[D]>0$ and $T<\infty$  a.s.\  due to Lemma \ref{mcd}.  This implies (XR). 
\end{proof}
By exponentiating $R$ we obtain from Theorem \ref{rtr} the following generalization of Proposition \ref{pro}.
\begin{cor}\label{grep}{\bf (General random exchange process)} 
Assume that  $T_1$ is bounded and $E[T_1]>0$. 
Moreover, suppose that there exists $\beta\in(2/3,1)$ such that  ${\displaystyle \lim_{x\to\infty}}x^{\beta}  P[W_1>x]=0$
or ${\displaystyle \liminf_{x\to\infty}x P[W_1>x]> E[T_1]}$. Let $y\in(0,\infty)$ be such that $P[W_1\le y]>0$. Then $R$ 
is recurrent iff 
 \[\sum_{n\ge 0}\prod_{m=0}^nP\left[W_1\le y+mE[T_1]\right]=\infty.\]
\end{cor}


\subsection{An application to random walks in random environments perturbed by cookies of maximal strength}
We consider the same version of excited random walks in random environment as Bauernschubert in \cite{Bau13}.
Let  $\om=(\om_x)_{x\in\Z}$ be an i.i.d.\ family of $(0,1)$-valued random variables and $Y=(Y_x)_{x\in \Z}$ be an i.i.d.\ family of $\N_0$-valued random variables such that $P[Y_0=0]>0$. We call $\om_x$ the environment at $x$ and $Y_x$ the number of cookies at $x$.  The random walk $\xi=(\xi_n)_{n\ge 0}$ in the random environment $\om$  perturbed by the cookies $Y$ is defined as follows. The walk starts at $\xi_0=0$. Upon
 any of the first $Y_x$ many visits to a site $x$ the walker reduces the number of cookies at that site by one and then moves in the next step deterministically to $x+1$. Upon the $(Y_x+1)$-st or any later visit to $x$, i.e.\ when there are no cookies left at $x$, the walker  jumps independently of everything else  with probability $\om_x$ to $x+1$ and with probability $1-\om_x$ to $x-1$. More formally, for all $n\ge 0$ and $z=\pm1$ a.s.
\[P[\xi_{n+1}=\xi_n+z\mid (\xi_{k})_{0\le k\le n},Y,\om]=\left\{\begin{array}{ll}1&\text{if $z=1, \#\{k\le n\mid \xi_k=\xi_n\}\le Y_{\xi_n}$}\\
                     \om_x&\text{if $z=1, \#\{k\le n\mid \xi_k=\xi_n\}> Y_{\xi_n}$}\\
                     1-\om_x&\text{if $z=-1, \#\{k\le n\mid \xi_k=\xi_n\}> Y_{\xi_n}$}.
                                             \end{array}\right.\]
The random walk $\xi$ is called transient to the right if $\xi_n\to\infty$ as $n\to\infty$, transient to the left if $\xi_n\to-\infty$ as $n\to\infty$, and recurrent if $\xi_n=0$ for infinitely many $n$. 
In the case without cookies, i.e.\ where $P[Y_0=0]=1$, we retrieve the classical one-dimensional random walk in random environment (RWRE).  It is  known that that under mild assumptions RWRE is a.s.\ recurrent iff $E[\ln\rho_0]=0$, where $\rho_0:=(1-\om_0)/\om_0$, and a.s.\ transient to the right (resp. left) iff $E[\ln\rho_0]<0$ (resp.\ $>0$), see e.g.\ \cite[Theorem 2.1.2]{Zei04}.

We consider the case  $E[\ln\rho_0]>0$ in which the underlying RWRE is transient to the left and ask how many cookies are needed in order to make this walk recurrent or even transient to the right. Using (\ref{BB1}), (\ref{BB2}),  and a well-known relationship between excursions of random walks and branching processes, Bauernschubert obtained in \cite{Bau13} the following result.
\begin{theo}\label{mon}{ \em  (\cite[Theorem 1.1]{Bau13})} 
Assume that the random variables $\om_x, Y_x\ (x\in\Z)$ are independent and let $E[|\ln\rho_0|]<\infty$, $E[\ln\rho_0]>0,$ and $E[\om_0^{-1}]<\infty$.
\begin{enumerate}
  \item[(a)] If $E[\ln_+Y_1]<\infty$ then $\xi$ is a.s.\ transient to the left.
  \item[(b)] If $E[\ln_+Y_1]=\infty$ and if $\limsup_{t\to\infty}t\cdot P[\ln Y_1>t]<E[\ln\rho_0]$, then 
   $\xi$ is a.s.\ recurrent.
   \item[(c)] If $\limsup_{t\to\infty}t\cdot P[\ln Y_1>t]>E[\ln\rho_0]$
   then $\xi$ is a.s.\ transient to the right.
 \end{enumerate}
\end{theo}
Replacing in the proof of this theorem (\ref{BB1}) and (\ref{BB2})  by Theorem \ref{rtr} we obtain the following complete characterization of recurrence/transience of $\xi$ in the so-called uniformly elliptic case where the transition probabilities $\om_x$ are bounded away from 0 and 1.
\begin{theorem}\label{rrr}
 Assume that $(\om_x,Y_x)_{x\in \Z}$ is independent  and that there is an $\eps>0$ such that a.s.\ $\om_0\in[\eps,1-\eps]$ and let $E[\ln\rho_0]>0$. 
 \begin{enumerate}
  \item[(a)] If $E[\ln_+Y_0]<\infty$ then $\xi$ is a.s.\ transient to the left.
  \item[(b)] If $E[\ln_+Y_0]=\infty$ and if  
  \begin{equation}\sum_{n\ge 0}\prod_{m=0}^n P\left[Y_0\le \exp\left(mE[\ln\rho_0]\right)\right]
   \label{exc}
  \end{equation}
is infinite  then $\xi$ is a.s.\ recurrent.
   \item[(c)] If the series in {\em (\ref{exc})} is finite
   then $\xi$ is a.s.\ transient to the right.
 \end{enumerate}
\end{theorem}

\section*{Appendix. Bounds for the case of random environment}

\begin{lemma}\label{dd} Let $\ga,K\in\N,$ $0<\kappa\le 1$ and  $\mathcal A\subseteq [0,\infty)^{d\times d}$.  
For $n\in\N_0$ set
 $\mathcal G_n:=\{A_1\ldots A_n: A_1,\ldots,A_n\in\mathcal A\}$ and
 $\mathcal G:=\bigcup_{n\ge 0}\mathcal G_n.$
 Assume that $\|A\|\le \ga$ for all $A\in\mathcal A$ and $\mathcal G_K\subseteq [\kappa,\infty)^{d\times d}$.
Then there is a constant $c=c(\ga,K,\kappa, d)$ such that 
  \begin{align}
  \|A\|\|x\|&\le c\|Ax\|&&\text{for all $A\in\mathcal G, x\in[0,\infty)^d$, }\label{wiede}\\
  \|A\|\|B\|&\le c\|AB\|&&\text{for all $A\in\mathcal G, B\in [0,\infty)^{d\times d}$,} \label{mann}\\
  \|A\|&\le c[A]_{1,1}&&\text{for all $n\ge K, A\in\mathcal G_n,$ and}\label{hopf}\\ 
\label{low}
 \kappa^{1/K}&\le \|A\|&&\text{for all $A\in\mathcal A$.} 
 \end{align}
\end{lemma}
\begin{proof} 
 For any matrix $A$ let $\mu(A):=\min_j\max_i [A]_{i,j}.$
The following two quantities are used  
 to measure the variation among the entries of $A$.
 \begin{align*}\delta_A&:=\|A\|_1/\mu(A)\in[1,\infty]&&\text{for $A\in [0,\infty)^{d\times d}\backslash\{0\}$ and}\\
   \Delta_A&:=\max\left\{\frac{[A]_{i,j}}{[A]_{i,k}},\frac{[A]_{i,j}}{[A]_{k,j}}: i,j,k\in\{1,\ldots,d\}\right\}\in[1,\infty)
   &&\text{for $A\in (0,\infty)^{d\times d}$.}
  \end{align*} 
 We first show the following relations.
\begin{align}
 \label{par0}
 \Delta_{AB}&\le \max\{\Delta_A, \Delta_B\}&&\text{for all $A, B\in (0,\infty)^{d\times d}$.}\\
 \label{par01}
 \Delta_{AB}&\le \Delta_A \delta_B&&\text{for all $A\in(0,\infty)^{d\times d},B\in [0,\infty)^{d\times d}\backslash\{0\}$.}\\
 \label{park}
 \delta_{AB}&\le \delta_A \delta_B&& \text{for all $A,B\in [0,\infty)^{d\times d}\backslash\{0\}$.}\\
  \delta_A&\le d\Delta_A&&\text{for all $A\in(0,\infty)^{d\times d}$}.\label{par4}
\end{align}
Statement (\ref{par0}) follows from the fact that for all $i,j,k\in\{1,\ldots,d\}$,
\begin{eqnarray}\nonumber
 \frac{[AB]_{i,j}}{[AB]_{i,k}}&=&\frac{\sum_n{[A]_{i,n}[B]_{n,j}}}{\sum_n{[A]_{i,n}[B]_{n,k}}}\le
\frac{\sum_n[A]_{i,n}\Delta_B [B]_{n,k}}{\sum_n{[A]_{i,n}[B]_{n,k}}} =\Delta_B\quad\text{and similarly}\\ 
 \frac{[AB]_{i,j}}{[AB]_{k,j}}&\le& \Delta_A.\label{sim2}
\end{eqnarray}
 To show (\ref{par01}) let $m$ and $k$  be such that $[B]_{m,k}=\max_{n}[B]_{n,k}=\mu(B)$. Then
 \[\frac{[AB]_{i,j}}{[AB]_{i,k}}
 \le\frac{\sum_n\Delta_A[A]_{i,m}[B]_{n,j}}{[A]_{i,m}[B]_{m,k}}\le\Delta_A \delta_B.
  \]
  Together with (\ref{sim2}) this proves (\ref{par01}).
  
For the proof of (\ref{park}) it suffices to show that
$\mu(AB)\ge\mu(A)\mu(B)$ since $\|AB\|_1\le\|A\|_1\|B\|_1$.
To this end, fix $1\le j\le d$, choose $k$ such that $[B]_{k,j}\ge\mu(B)$ and
$m$ such that $[A]_{m,k}\ge\mu(A)$. Then $\max_i[AB]_{i,j}\ge [AB]_{m,j}\ge [A]_{m,k}[B]_{k,j}\ge \mu(A)\mu(B)$. Taking the minimum over $j$ yields (\ref{park}).

To prove (\ref{par4}) let $k$ be such that $\max_i[A]_{i,k}=\mu(A)$. Then
\[
 \delta_A=\frac{\max_j\sum_i[A]_{i,j}}{\mu(A)}\le \frac{\Delta_A\sum_i[A]_{i,k}}{\max_i[A]_{i,k}}\le d\Delta_A.
\]
This concludes to proof of (\ref{par0})--(\ref{par4}).
Next we show that
 \begin{eqnarray}
  \sup\{\Delta_A: A\in \mathcal G_n, n\ge K\}&<&\infty\quad\text{and}\label{ger}\\
  \label{pln}
  \sup\{\delta_A: A\in\mathcal G\}&<&\infty.
 \end{eqnarray}
First note that 
 $\con{bri}:=\sup\{\delta_A: A\in\mathcal A\}<\infty.$
Indeed, let $A\in\mathcal A$ and $B\in\mathcal G_{K-1}$. Choose $j$ such that $\max_i[A]_{i,j}=\mu(A)$. Since $BA\in\mathcal G_K$ we have
$\kappa\le[BA]_{1,j}=\sum_i[B]_{1,i}[A]_{i,j}\le \|B\|\mu(A)$
and consequently,
$\delta_A\le \|A\|_1\|B\|/\kappa\le d\ga^{K}/\kappa$.

Second, due to $\mathcal G_K\subseteq [\kappa,\infty)^{d\times d}$, no element of $\mathcal A$ has a column of zeros. Hence, $\mathcal G_n\subseteq (0,\infty)^{d\times d}$ for all $n\ge K$. Therefore, if we let  $K\le n=mK+r$ with $m\ge 1$ and $0\le r<K$ then for all 
$A_1,\ldots,A_n\in\mathcal A$,
\begin{eqnarray}\nonumber
 \Delta_{A_1\ldots A_n}&\overset{(\ref{par01})}\le&\Delta_{A_1\ldots A_{mK}}\delta_{A_{mK+1}\ldots A_n}\\  
 \label{zh}
 &\overset{(\ref{par0}),(\ref{park})}\le& \max_{i=0}^{m-1}\Delta_{A_{iK+1}\ldots A_{(i+1)K}}\delta_{A_{mK+1}}\ldots \delta_{A_n}\ \le\ \ga^{K}\kappa^{-1}c_{\ref{bri}}^K=:\con{sz},
\end{eqnarray}
where we used in the last step that $\Delta_{B}\le \|B\|/\kappa\le \ga^{K}/\kappa$  for any $B\in\mathcal G_K$.
This implies (\ref{ger}). 
Moreover,  (\ref{park}) implies
$\delta_A\le  c_{\ref{bri}}^{K}$ for all $A\in\mathcal G_n, n\le K,$ and (\ref{par4}) and (\ref{zh}) imply $\delta_A\le  dc_{\ref{sz}}$ for all $A\in\mathcal G_n, n\ge K$. Together this yields (\ref{pln}).

For the proof of  the first claim of the Lemma, (\ref{wiede}), let $k$ be such that $\|x\|=x_k$. Then for all $A\in\mathcal G$,
\begin{eqnarray*}
\|Ax\|&=&\max_i\sum_j[A]_{i,j}x_j \ge \max_i[A]_{i,k}x_k = \|x\|\max_i[A]_{i,k}\\
&\ge& \|x\|\min_j\max_i[A]_{i,j}=\frac{\|A\|_1\|x\|}{\delta_A}\ge \frac{\|A\|\|x\|}{d\delta_A}.
\end{eqnarray*}
Along with (\ref{pln}) this implies (\ref{wiede}). The second claim, (\ref{mann}), follows from (\ref{wiede}) and the definition of the matrix norm $\|\cdot\|$.
For the proof of (\ref{hopf}) let $n\ge K$ and $ A\in\mathcal G_n$. Then
\[
\|A\| =\max_k\sum_{\ell}[A]_{k,\ell}\le \Delta_A\sum_{\ell} [A]_{1,\ell}\le\Delta_A^2\sum_{\ell} [A]_{1,1}= d\Delta_A^2[A]_{1,1}.
\]
This along with (\ref{ger}) implies (\ref{hopf}). The last claim, (\ref{low}), follows from $\kappa\le\|A^K\|\le\|A\|^K$.
\end{proof}
The following result provides bounds on the extinction time of multitype branching process in varying environment. The easy bound is standard and based on a first moment method, i.e.\ Markov's inequality. To the best of our knowledge the opposite bound appeared first in a similar form in  \cite[Theorem 1]{Agr75}. We prove it by the second moment method. For precise asymptotics under different assumptions see e.g.\ \cite{JS67}, \cite{Dya08}.
\begin{prop}\label{ex}{\bf (Bounds on extinction time of multitype branching process in varying environment)}
Fix $\psi=(\psi_n)_{n\ge 1}=((\psi_n^j)_{j=1,\ldots,d})_{n\ge 1}\in (\Phi^d)^{\N}$. Let $(U_{n,k}^j)_{n,k\ge 1; j\in\{1,\ldots,d\}}$ be an i.i.d.\ family of random variables which are distributed uniformly  on $[0,1]$.
Let $\xi_{n,k}^{i,j}:=[\psi_n^j(U_{n,k}^j)]_i$. Fix $s\in\{1,\ldots,d\}$  and define the branching process $(B_n^{})_{n\ge 0}$ in the environment $\psi$ starting at time 0 with one individual of type $s$ as follows. Set $B^{}_0:=e_s$ and define recursively for all $n\ge 1$,
\begin{equation}\label{fol}
 B^{}_{n}:=\left(\sum_{j=1}^d\sum_{k=1}^{[B^{}_{n-1}]_j}\xi_{n,k}^{i,j}\right)_{i=1,\ldots,d}.
\end{equation}
Define the matrices $A_n:=
\left(E\left[\xi_{n,1}^{i,j}\right]\right)_{i,j=1,\ldots,d}$, $n\ge 1,$ and suppose that $\ga,K\in\N,\kappa\in(0,1]$ and $\mathcal A:=\{A_n\mid n\ge 1\}$ satisfy the assumptions of Lemma \ref{dd}.  
Denote for  $n\ge 1$ and $j=1,\ldots,d$ by $\mathcal V_n^j$ the covariance matrix of the vector $(\xi_{n,1}^{i,j})_{i=1,\ldots,d}$ and suppose that
$\con{rip}:= \sup_{n\ge 1, j=1,\ldots,d}\|\mathcal V_n^j\|<\infty.$
Then there is a constant $c_{\ref{dee}}=c_{\ref{dee}}(\ga,K,\kappa,d,c_{\ref{rip}})$ such that for all  $n\ge 1$,
  \[
 c_{\ref{dee}}\frac{\|A_{n}\ldots A_1\|}{\sum_{k=1}^{n}\|A_{n}\ldots A_{k}\|}\le P[B^{}_n\ne 0]\le
 d\|A_{n}\ldots A_1\|.
  \]
\end{prop}
\begin{proof} 
It follows from (\ref{fol}) that $E[B^{}_n]=A_nE[B^{}_{n-1}]$ for all $n\ge 1$, see e.g.\ \cite[Chapter II, (4.1)]{Har63}. Therefore,
$E[B_n]=A_n\ldots A_1e_s$ for all $n\ge 0$. Thus
\[P[B_n\ne 0]=P[\|B_n\|_1\ge 1]\le E[\|B_n\|_1]=\|A_n\ldots A_1e_s\|_1\le d\|A_n\ldots A_1\|.
\]
For the lower bound set $\mathbf C_n:=\left(E[[B^{}_n]_i[B^{}_n]_j]\right)_{i,j=1,\ldots,d}$.
By the second moment method
\begin{eqnarray}\nonumber
 P[B^{}_n\ne 0] &=& P\left[\|B^{}_n\|>0\right]\ge\frac{\left(E[\|B^{}_n\|]\right)^2}{E\left[\|B^{}_n\|^2\right]}\ge
 \frac{\|E[B^{}_n]\|^2}{E\left[\max_i [B^{}_n]_i^2\right]}\\
 &\ge&\frac{\|A_{n}\ldots A_1e_s\|^2}{\sum_{i=1}^dE\left[[B^{}_n]_i^2\right]} 
\overset{(\ref{wiede})}\ge 
\con{bl}\frac{(\|A_{n}\ldots A_1\|\|e_s\|)^2}{\max_{i=1}^dE\left[[B^{}_n]_i^2\right]}
 \ge  c_{\ref{bl}}\frac{\|A_{n}\ldots A_1\|^2}{\|\mathbf C_n\|}.\label{cga}
\end{eqnarray}
By \cite[Chapter II, (4.2)]{Har63} for all $n\ge 1$,
\begin{eqnarray*}
  \mathbf C_{n}&=&A_{n}\mathbf C_{n-1}A_{n}^T+\sum_{j=1}^dE\left[[B^{}_{n-1}]_j\right]\mathcal V_{n}^j\\
  &=&A_{n}\ldots A_1\mathbf C_0A_1^T\ldots A_{n}^T+\sum_{k=1}^{n}A_{n}\ldots A_{k+1}\left(\sum_{j=1}^dE[[B^{}_{k-1}]_j]\mathcal V_k^j\right)A_{k+1}^T\ldots A_{n}^T
\end{eqnarray*}
by induction.
Consequently, 
  \begin{eqnarray*}
\|\mathbf C_{n}\|&\le&\con{gst}\|A_{n}\ldots A_1\|^2
+\ \sum_{k=1}^{n}\|A_{n}\ldots A_{k+1}\|\left(\sum_{j=1}^dE[[B^{}_{k-1}]_j]\|\mathcal V_k^j\|\right)\|(A_{n}\ldots A_{k+1})^T\|\nonumber\\
&\overset{(\ref{low})}\le&c_{\ref{gst}}\|A_{n}\ldots A_1\|^2+\con{wo}\sum_{k=1}^{n}\|A_{n}\ldots A_{k+1}\|^2\|A_k\|\|E[B^{}_{k-1}]\|_1\\
&\overset{(\ref{wiede})}\le&c_{\ref{gst}}\|A_{n}\ldots A_1\|^2+\con{wer}\sum_{k=1}^{n}\|A_{n}\ldots A_{k+1}\|\|A_{n}\ldots A_{k}E[B^{}_{k-1}]\|\\
&\le&c_{\ref{gst}}\|A_{n}\ldots A_1\|^2+
   c_{\ref{wer}}\|A_{n}\ldots A_{1}\|\|E[B^{}_0]\|\sum_{k=1}^{n}\|A_{n}\ldots A_{k+1}\|\\
   &\le& \con{waru}\|A_{n}\ldots A_1\| \sum_{k=1}^{n}\|A_{n}\ldots A_{k}\|.
 \end{eqnarray*}
  Substituting this into (\ref{cga}) yields the claim.
 \end{proof}
\begin{lemma}{\bf (Concentration inequality)} \label{mcd}
Assume {\em (BD1)} and {\em (PR)}.
Then there are constants $c_{\ref{coo}}$ and $c_{\ref{co}}$ depending on $(d,\vk,\ga_1,K)$ such that for all $n\ge 0$ and  $t\in(0,\infty)$,
\begin{equation}P\left[|S_n-\la n|\ge t\right]\le c_{\ref{coo}}\exp\left(-c_{\ref{co}}t^2/n\right).
 \label{cl}
\end{equation}
\end{lemma}
\begin{proof}
Denote by $\mathcal A\subseteq [0,\infty)^{d\times d}$ the support of $A_1$.
Due to (BD1) and (PR) the assumptions of Lemma \ref{dd} are satisfied.
First we  show the existence of  $\con{coa}>0$ such that for all $n\ge 0$ and  $t>0$,
\begin{equation}P\left[|S_n-E[S_n]|\ge t\right]\le 2\exp\left(-c_{\ref{coa}}t^2/n\right).\label{mc1}
\end{equation}
 Let $f(B):=-\ln\|B_1\ldots B_n\|$ for any $n\ge 1$ and  $B=(B_1,\ldots,B_n)\in \mathcal A^n$.
  Suppose  $B,B'\in \mathcal A^n$ differ only in a single coordinate, say the $k$-th one. Then 
\begin{eqnarray*}
f(B)-f(B')&\le&\ln\left(\|B_1\ldots B_{k-1}\| \|B'_k\| \|B_{k+1}\ldots B_{n}\|\right)\\
&&-\ln\left(c^2\|B_1\ldots B_{k-1}\| \|B_k\| \|B_{k+1}\ldots B_{n}\|\right)
\le  \con{con}
\end{eqnarray*}
due to submultiplicativity, $\|B'_k\|\le \ga_1$, (\ref{mann}), and (\ref{low}).
By symmetry, $|f(B)-f(B')|\le c_{\ref{con}}$. Now (\ref{mc1}) follows from McDiarmid's inequality \cite[Lemma (1.2)]{McD89}.
By the subadditive ergodic theorem,
\begin{equation}\label{elle}
\sup_{n\ge 1}\frac{E\left[S_n\right]}n\ =\ \la\ =\ \lim_{n\to\infty}\frac{E[S_{nK}]}{nK}\ \le\ \liminf_{n\to\infty}\frac{E[-\ln [A_1\ldots A_{nK}]_{1,1}]}{nK}.
\end{equation}
Since $[AB]_{1,1}\le [A]_{1,1}[B]_{1,1}$ for any $A,B\in[0,\infty)^{d\times d}$, another application of the subadditive ergodic theorem yields that the right most side of (\ref{elle}) is equal to
\begin{equation}\label{war}
\inf_{n\ge 1}\frac{E[-\ln [A_1\ldots A_{nK}]_{1,1}]}{nK}\overset{(\ref{hopf})}\le \inf_{n\ge 1}\frac{\con{bat}+E[S_{nK}]}{nK}.
\end{equation}
By submultiplicativity, for all $0\le r<K$ and $n\ge 1,$
$E[S_{nK+r}]\ge E[S_{nK}]+rE[S_1]\ge E[S_{nK}]-K\ln \ga_1$ due to (BD1). Consequently, the right hand side of (\ref{war}) is at most 
\[\inf_{0\le r<K}\inf_{n\ge 1}\frac{\con{bath}+E[S_{nK+r}]}{nK}\le \inf_{n>K}\frac{c_{\ref{bath}}+E[S_n]}{n-K}.
\]
Together with (\ref{elle}) this implies that $|\la n-E[S_n]|\le \la K+c_{\ref{bath}}$ for all $n>K$. 
The claim now follows from (\ref{mc1}).
\end{proof}
\begin{lemma}\label{W}
 Under the assumptions of Lemma \ref{mcd} suppose that $\la>0$ such that $\si:=\sum_{i\ge 1}\|A_1\ldots A_i\|<\infty$ a.s.. Let $\zeta>0$ and $0<\eta\le 1$  and set
 \[
  T:=\inf\left\{n\ge 0\mid\forall i>n:\ S_i\le \la i+\zeta i^\eta-\ln\si\right\}.
 \]
Then there are constants $c_{\ref{et}}, c_{\ref{et2}}$ depending on $(d,\ga_1,K,\vk,\la,\zeta,\eta)$ such that 
\[P[T=n]\le c_{\ref{et}}\exp\left(-c_{\ref{et2}}n^{2\eta-1}\right)\quad\text{for all $n\ge 0$.}\]
\end{lemma}
\begin{proof}
 Let $\con{abc}:=\left(\sum_{i\ge 1}e^{-\la i/2}\right)^{-1}$. Then for all $t>c_{\ref{abc}}^{-1}$,
 \begin{eqnarray}\nonumber
  P[\si\ge t]&\le&\sum_{i\ge 1}P\left[\|A_1\ldots  A_i\|\ge c_{\ref{abc}}e^{-\la i/2}t\right]\ \le\ \sum_{i\ge 1}P[|S_i-\la i|\ge \la i/2+\ln (c_{\ref{abc}}t)]\\ 
  &\overset{(\ref{cl})}\le&\sum_{i\ge 1}c_{\ref{coo}}e^{-c_{\ref{co}}\left(\la i/2+\ln (c_{\ref{abc}}t)\right)^2/i} \label{cdf}
  \ \le\  c_{\ref{coo}}\sum_{i\ge 1}e^{-c_{\ref{co}}\left(\la^2 i/4+\la \ln (c_{\ref{abc}}t)\right)}\ =\ \con{art}t^{-\con{vb}}.
 \end{eqnarray}
 Therefore, for all $n\ge 1$,
 \begin{eqnarray*}
  P[T=n]&\le& P\left[\la n+\zeta n^\eta-\ln \si\le S_n\right]\\
  &\le& P\left[\la n+\zeta n^\eta-\ln \si\le S_n\le \la n+(\zeta/2) n^\eta\right]  + 
  P\left[S_n\ge \la n+(\zeta/2) n^\eta\right]\\
  &\overset{(\ref{cl})}\le&P[\si\ge e^{(\zeta/2) n^\eta}]+c_{\ref{coo}}e^{-c_{\ref{et2}}n^{2\eta-1}}
  \overset{(\ref{cdf})}\le c_{\ref{art}}e^{-\con{qua}n^\eta}+c_{\ref{coo}}e^{-c_{\ref{et2}}n^{2\eta-1}}.
 \end{eqnarray*}
  Since $\eta\le 1$ this yields the claim.
\end{proof}

{\bf Acknowledgment:} The author is grateful to Martin M\"ohle for pointing out \cite{Lam70}
and thanks him and  Elena Kosygina for helpful discussions.

\end{document}